\documentclass[a4paper,10pt]{amsart}

\usepackage[english]{babel}
\usepackage[utf8]{inputenc}

\usepackage{amssymb}
\usepackage{amsmath}
\usepackage{amsthm}
\usepackage{bbm}

\usepackage{enumerate}
\usepackage{tikz}
\usepackage{marginnote}

\newcommand{\bbC}{\mathbb{C}}

\newcommand{\bbN}{\mathbb{N}}

\newcommand{\bbQ}{\mathbb{Q}}
\newcommand{\bbR}{\mathbb{R}}

\newcommand{\bbT}{\mathbb{T}}

\newcommand{\bbZ}{\mathbb{Z}}

\DeclareMathOperator{\one}{\mathbbm{1}}

\newcommand{\argument}{\mathord{\,\cdot\,}}

\newcommand{\norm}[1]{\left\lVert #1 \right\rVert}
\newcommand{\modulus}[1]{\left\lvert #1 \right\rvert}
\DeclareMathOperator{\dom}{dom}

\newcommand{\spec}{\sigma}
\newcommand{\Res}{\mathcal{R}}

\newcommand{\pnt}{{\operatorname{pnt}}}

\newcommand{\pntSpec}{\spec_{\pnt}}

% ...

\theoremstyle{definition}
\newtheorem{definition}{Definition}[section]
\newtheorem{construction}[definition]{Construction}
\newtheorem{remark}[definition]{Remark}

\theoremstyle{plain}
\newtheorem{proposition}[definition]{Proposition}

\newtheorem{theorem}[definition]{Theorem}

\numberwithin{equation}{section}

\begin{document}

\title[The spectrum of irreducible operators]{A note on the spectrum of irreducible operators and semigroups}
\author{Jochen Gl\"uck}
\address{Jochen Gl\"uck, Universität Passau, Fakultät für Informatik und Mathematik, 94032 Passau, Germany}
\email{jochen.glueck@uni-passau.de}
\subjclass[2010]{47B65; 47A10}
\keywords{Irreducible operator; irreducible operator semigroup; peripheral spectrum; cyclicity}
\dedicatory{Dedicated with great pleasure to Rainer Nagel on the occasion of his 80th birthday}
\date{\today}
\begin{abstract}
	Let $T$ denote a positive operator with spectral radius $1$ on, say, an $L^p$-space.
	A classical result in infinite dimensional Perron--Frobenius theory says that,
	if $T$ is irreducible and power bounded, 
	then its peripheral point spectrum is either empty or a subgroup of the unit circle.
	
	In this note we show that the analogous assertion for the entire peripheral spectrum fails.
	More precisely, for every finite union $U$ of finite subgroups of the unit circle
	we construct an irreducible stochastic operator on $\ell^1$
	whose peripheral spectrum equals $U$.
	
	We also give a similar construction for the $C_0$-semigroup case.
\end{abstract}

\maketitle

\section{Introduction} \label{section:introduction}

\subsection*{Main result}

Consider a positive operator $T$ with spectral radius $1$ on an $L^p$-space
($p \in [1,\infty]$) or, more generally, on a complex Banach lattice $E$. 
Assume that $T$ is power-bounded (i.e., that $\sup_{n \in \bbN_0} \norm{T^n} < \infty$),
has spectral radius $1$ and is 
irreducible (i.e., $T$ does not leave any closed ideal invariant,
except for $\{0\}$ and $E$; see e.g.\ \cite[Definition~4.2.1(1)]{Meyer-Nieberg1991}).
Then the \emph{peripheral point spectrum} of $T$
-- i.e., the set of all eigenvalues of $T$ of modulus $1$ --
is either empty or a subgroup of the complex unit circle $\bbT$.
This result, which is a strong generalisation 
of the corresponding Perron--Frobenius result in finite dimensions, 
was proved by Lotz \cite[Theorem~5.2(2)]{Lotz1968}
(for an English presentation of this result
see, for instance, \cite[Theorem~V.5.2 and Lemma~V.4.8]{Schaefer1974}).

For the \emph{peripheral spectrum} of $T$ 
-- i.e., the set of all spectral values of modulus $1$ --
no such result is known on general Banach lattices 
(but see Remark~\ref{rem:continuous-functions} 
for spaces of continuous functions). 
And indeed, the purpose of this note is to demonstrate 
that the peripheral spectrum of an irreducible operator 
need not be a subgroup of the unit circle, in general.
More precisely, we show:

\begin{theorem}
	\label{thm:main-result}
	Let $U \not= \emptyset$ be a finite union of finite subgroups of the complex unit circle.
	Then there exists an irreducible stochastic operator $T$ on $\ell^1$
	with peripheral spectrum $U$.
\end{theorem}

By \emph{stochastic} we mean that, for each $0 \le f \in \ell^1$,
we have $Tf \ge 0$ and $\norm{Tf} = \norm{f}$.
In order to prove the theorem, 
we explicitly construct an operator $T$ with the desired properties
in Section~\ref{section:proof-of-main-result}.
An analogous result for $C_0$-semigroups rather than for single operators
is given in Section~\ref{section:semigroup-version}.

\begin{remark}
	\label{rem:continuous-functions}
	Complementary to Theorem~\ref{thm:main-result}, there is a positive result 
	on spaces of continuous functions:
	consider the space $C(K)$ of continuous scalar-valued functions 
	on a compact Hausdorff space $K$; 
	if $T$ is a positive and irreducible operator on $C(K)$ 
	and if the constant function $\one$ is a fixed vector of $T$,
	then the peripheral spectrum of $T$ is indeed a subgroup of the complex unit circle $\bbT$.
	This was proved by Schaefer in \cite[Theorem~7]{Schaefer1968}.
\end{remark}

\subsection*{Related results and literature}

For an overview of classical Perron--Frobenius type results on 
infinite-dimensional Banach lattices we refer to the book chapters
\cite[Sections~V.4 and~V.5]{Schaefer1974} or \cite[Chapter~4]{Meyer-Nieberg1991}, 
or to the survey article \cite{Grobler1994}.

If one considers merely positive rather than irreducible operators,
there is no reason to expect the peripheral (point) spectrum to be a subgroup
of the unit circle -- consider, for instance, 
the direct sum of two cyclic permutation matrices in dimensions $2$ and $3$.
However, there is a weaker notion of symmetry that is satisfied 
by the peripheral spectrum of many positive operators:
we call a subset of the unit circle $\bbT$ \emph{cyclic} 
if it is a union of subgroups of $\bbT$.
It was shown by Krieger \cite[Folgerungen~2.2.1(b) and~2.2.2(b)]{Krieger1969}
and Lotz \cite[Theoreme~4.7, 4.9 and~4.10]{Lotz1968}
that, under certain growth assumptions, 
the peripheral spectrum of a positive operator on a Banach lattices is cyclic
(the results can be found in English 
in \cite[Theorem~V.4.9 and its corollary]{Schaefer1974}).
Whether the aforementioned growth assumptions can be dropped
is a long open problem in spectral theory.
A recent overview of this problem, along with several partial results, 
can be found in \cite{Glueck2018}.

The cyclicity result of Krieger and Lotz also explains the assumed form
of the set $U$ in Theorem~\ref{thm:main-result}:
the positivity of $T$ (together with its contractivity) already implies
the the peripheral spectrum must be a union of subgroups of $\bbT$.

The related question whether the peripheral \emph{point} spectrum
of a positive operator is cyclic, is very subtle, 
and its answer depends on the precise properties of the operator under consideration
as well as on the geometric properties of the underlying space.
For various results and counterexamples, as well as for further references,
we refer to \cite[Sections~5 and~6]{Glueck2016}.

\subsection*{Notation and terminology}

We use the convention $\bbN := \{1,2,3\dots\}$;
the complex unit circle is denoted by $\bbT$.

We call a bounded linear operator $T$ on a complex Banach lattice $E$ 
\emph{positive} if $Tf \ge 0$ for each $0 \le f \in E$.
If $\lambda \in \bbC$ is in the resolvent set of a linear operator $T$,
we denote the resolvent of $T$ at $\lambda$ by $\Res(\lambda,T) := (\lambda - T)^{-1}$.

\section{Proof of the main result} 
\label{section:proof-of-main-result}

Let $G_1, \dots, G_n \subseteq \bbT$ be finite subgroups such that
$U = G_1 \cup \dots \cup G_n$.
We denote the cardinality of $G_k$ by $d_k$ and we set $d := d_1 + \dots + d_n$. 
It is very easy to construct a finite (column) stochastic matrix with spectrum $U$: 
for each $k \in \{1,\dots,n\}$ let $P_k \in \bbR^{d_k \times d_k}$ 
denote a cyclic permutation matrix. 
Then the spectrum of $P_k$ is $G_k$, so the permutation matrix
\begin{align*}
	P := P_1 \oplus \dots \oplus P_n \in \bbR^{d \times d}
\end{align*}
has spectrum $U$.
The point here is, of course, that $P$ is not irreducible.
In order to get an irreducible operator, 
we now take a direct sum of infinitely many copies of $P$ and slightly perturb it.
The main difficulty is then to check that the perturbed operator 
does not have spectral values in $\bbT \setminus U$. 
Here is the detailed construction:

\begin{construction}
	\label{construction:single-operator}
	\emph{The space:} 
	Endow $\bbC^d$ with the $1$-norm $\norm{\argument}_1$
	and consider the space
	\begin{align*}
		E := \{f = (f_0, f_1, f_2, \dots): \; f_0 \in \bbC, \; f_1,f_2, \dots \in \bbC^d \text{ and } \norm{f}_E < \infty\},
	\end{align*}
	where
	\begin{align*}
		\norm{f}_E := \modulus{f_0} + \sum_{n=1}^\infty \norm{f_n}_1.
	\end{align*}
	Obviously, $E$ is isometrically lattice isomorphic to $\ell^1$.
	
	\emph{The operator:}
	Choose a sequence of numbers $(q_n)_{n \in \bbN}$ in $(0, \frac{1}{2}]$
	such that 
	$$
		\sum_{n=1}^\infty dq_n = 1.
	$$
	Clearly, we have $q_n \to 0$ as $n \to \infty$.
	
	We use the symbol $\one \in \bbR^d$ to denote the vector in $\bbR^d$ 
	whose entries are all equal to $1$,
	and we consider the vectors 
	\begin{align*}
		e = (1,0,0,\dots) \in E 
		\qquad \text{and} \qquad 
		q = (0, q_1\one, q_2\one, \dots) \in E
	\end{align*}
	Both vectors $e$ and $q$ have norm $1$.
	
	Let us define the operator $T: E \to E$ by the formula
	\begin{align*}
		T
		\begin{pmatrix}
			f_0 \\
			f_1 \\ 
			f_2 \\
			\vdots
		\end{pmatrix}
		=
		\begin{pmatrix}
			\sum_{n=1}^\infty q_n \langle \one, f_n\rangle \\
			(1-q_1)P f_1  +  f_0 q_1 \one \\
			(1-q_2)P f_2  +  f_0 q_2 \one \\
			\vdots
		\end{pmatrix}.
	\end{align*}
	Note that we can write $T$ in the form $T = S + e \otimes q + q \otimes e$,
	where the operator $S: E \to E$ has the block diagonal form
	\begin{align*}
		S = 
		\begin{pmatrix}
			0 &          &          &        \\
			  & (1-q_1)P &          &        \\
			  &          & (1-q_2)P &        \\
			  &          &          & \ddots
		\end{pmatrix},
	\end{align*}
	and where the rank-$1$ operators $e \otimes q$ and $q \otimes e$ are given by
	\begin{align*}
		(e \otimes q)f := \langle e, f \rangle q 
		\qquad \text{and} \qquad
		(q \otimes e) f := \langle q, f \rangle e
	\end{align*}
	for all $f \in E$ 
	(note that this makes sense since both sequences $e$ and $q$ 
	can be seen as elements of $E \simeq \ell^1$ and as elements of $E' \simeq \ell^\infty$).
\end{construction}

Now that we have constructed our operator $T$ on $E$,
let us check that it satisfies all the desired properties.
Obviously, $T$ is positive, 
and since all $q_n$ are $> 0$, one readily sees that $T$ is irreducible, too.
The fact that $T$ is stochastic follows from the equality $\sum_{n=1}^\infty d q_n = 1$
and from the fact that the matrix $P$ is (column) stochastic.

It is also not difficult to see that each number in $U$ is a spectral value of $T$:

\begin{proposition}
	\label{prop:approximate-eigenvalue}
	Let $\lambda \in U$. 
	Then $\lambda$ is an approximate eigenvalue of $T$.
\end{proposition}
\begin{proof}
	Clearly, $\lambda$ is a spectral value of the matrix $P$. 
	Let $z \in \bbC^d$ be a corresponding eigenvector of norm $\norm{z}_1 = 1$.
	
	Now we can easily construct an approximate eigenvector for $T$:
	for each $n \in \bbN$, let $z^{(n)} \in E$ be the vector
	that is equal to $z$ at the $n$-th component, and $0$ everywhere else.
	This vector has norm $1$ in $E$.
	
	Consider the vector $Tz^{(n)}$:
	its $0$-th component is $q_n \langle \one, z \rangle$,
	its $n$-th component is $\lambda (1-q_n)z$,
	and all other components are $0$.
	Thus, $(T-\lambda)z^{(n)} \to 0$ in $E$ as $n \to \infty$.
	This proves that $\lambda$ is an approximate eigenvalue of $T$.
\end{proof}

It only remains to show that no other unimodular number is in the spectrum of $T$.
This is a bit more involved since it is, in general, 
not easy to show that a given complex number is not in the spectrum of a given operator.
What will save us in our concrete situation is the formula
$$
	T = S + e \otimes q + q \otimes e,
$$
which says that $T$ is an additive perturbation of the very simple operator $S$
by two rank-$1$ operators;
for rank-$1$ perturbations, there is a variant of the famous Sherman--Morrison formula
which will allow us to retrieve very precise information about the spectrum of $T$.
But first, we analyse the spectrum of the simpler operator $S$:

\begin{proposition}
	\label{prop:spectrum-of-S}
	Let $\lambda \in \bbT \setminus U$.
	Then $\lambda$ is in the resolvent set of $S$
\end{proposition}
\begin{proof}
	Clearly, $\lambda$ is in the resolvent set of $(1-q_n)P$ for each $n$.
	Due to the block diagonal structure of $S$, 
	it suffices to prove that the resolvents $\Res\big(\lambda,(1-q_n)P\big)$
	are uniformly bounded as $n$ varies.
	
	Fortunately, this is easy:
	since $\lambda$ is not in the spectrum of $P$, 
	there exists a constant $M \ge 0$ such that
	\begin{align*}
		\norm{\Res(r \lambda, P)} \le M
	\end{align*}
	for all $r \in [1,2]$. 
	Since all $q_n$ are in the interval $(0,\frac{1}{2}]$, this yields
	\begin{align*}
		\norm{\Res\big(\lambda, (1-q_n)P\big)}
		= \frac{1}{1-q_n} \norm{\Res\Big(\frac{\lambda}{1-q_n}, P\Big)} \le 2M
	\end{align*}
	for all $n \in \bbN$.
\end{proof}

In order to show that the perturbation by 
$e \otimes q$ and $q \otimes e$ does not destroy the spectral properties
on the unit circle, we now use the following version 
of the Sherman--Morrison formula:

\begin{proposition}
	\label{prop:sherman-morrison-woodbury}
	Let $A: X \supseteq \dom{A} \to X$ be a closed linear operator on a complex Banach space $X$,
	let $w \in X$ and $\varphi \in X'$.
	Let $\lambda \in \bbC$ be in the resolvent set of $A$.
	
	Then $\lambda$ is in the resolvent set of $A + \varphi \otimes w$
	if and only if $\langle \varphi, \Res(\lambda,A) w\rangle \not= 1$;
	in this case, the resolvent of $A + \varphi \otimes w$ at $\lambda$
	is given by the formula
	\begin{align}
		\label{eq:prop:sherman-morrison-woodbury:formula}
		\Res(\lambda, A + \varphi \otimes w)
		=
		\Res(\lambda, A) + 
		\frac{1}{1 - \langle \varphi, \Res(\lambda,A) w\rangle}
		\Res(\lambda, A) (\varphi \otimes w) \Res(\lambda,A).
	\end{align}
\end{proposition}

The reason why we formulated the proposition
for closed rather than merely for bounded operators is 
that we will also employ this result for semigroup generators in the subsequent section.

In finite dimensions (and for $\lambda = 0$) Proposition~\ref{prop:sherman-morrison-woodbury}
is a classical result in matrix analysis; 
for an historical overview we refer to \cite[Section~1]{Hager1989}.
In infinite dimensions the proposition can, for bounded operators $A$, be
found in \cite[Theorem~1.1]{Foguel1960}; for the case of unbounded $A$
we refer to \cite[Lemma~1.1]{ArendtBatty2006} or \cite[Proposition~A.1]{DanersGlueck2018}. 

We will now use Proposition~\ref{prop:sherman-morrison-woodbury}
to show that $T$ has no spectral values in $\bbT \setminus U$.
Since $T$ is a rank-$2$ perturbation of $S$, 
we have to employ the proposition twice.
To perform all the necessary computations, 
the following simple observation about complex numbers is useful:

\begin{proposition}
	\label{prop:circles}
	Let $p \in [0,1)$. 
	For every complex $\lambda \in \bbT \setminus \{1\}$ we have
	\begin{align*}
		p < \modulus{\lambda - 1 + p}.
	\end{align*}
\end{proposition}
\begin{proof}
	It is easy to see this geometrically:
	the circle $\bbT - 1 + p$ (which is centered at $-1+p$) 
	is located outside the circle $p \bbT$,
	and they only intersect in the point $p$.
\end{proof}

Now, finally, we can conclude the proof of our main result, Theorem~\ref{thm:main-result},
by showing the following proposition:

\begin{proposition}
	\label{prop:resolvent-set-on-unit-circle}
	Let $\lambda \in \bbT \setminus U$. 
	Then $\lambda$ is in the resolvent set of $T$.
\end{proposition}
\begin{proof}
	From Proposition~\ref{prop:spectrum-of-S} we know 
	that $\lambda$ is in the resolvent set of $S$.
	Now, let us first show that $\lambda$ is in the resolvent set of $S + e \otimes q$,
	and let us also compute the resolvent of this operator at $\lambda$.
	
	The resolvent of $S$ at $\lambda$ can be written down in block diagonal form,
	and from this, we obtain $\langle e, \Res(\lambda,S) q \rangle = 0$.
	Thus, by Proposition~\ref{prop:sherman-morrison-woodbury},
	$\lambda$ is indeed in the resolvent set of $S + e \otimes q$, and we have
	\begin{align*}
		\Res(\lambda, S + e \otimes q) 
		= 
		\Res(\lambda, S) + \Res(\lambda, S) (e \otimes q) \Res(\lambda, S).
	\end{align*}
	Next, we perturb $S + e \otimes q$ by $q \otimes e$,
	and we apply Proposition~\ref{prop:sherman-morrison-woodbury} a second time
	in order to see that $\lambda$ is in the resolvent set of
	$T = S + e \otimes q + q \otimes e$: we have
	\begin{align*}
		\langle q, \Res(\lambda,S + e \otimes q) e \rangle 
		& =
		\langle q, \Res(\lambda, S) e \rangle 
		+
		\left\langle
			q, \Res(\lambda, S) (e \otimes q) \Res(\lambda, S) e
		\right\rangle
		\\
		& =
		\langle e, \Res(\lambda, S) e \rangle  \cdot  \langle q, \Res(\lambda, S) q \rangle
		=
		\frac{1}{\lambda}  \cdot  \langle q, \Res(\lambda, S) q \rangle;
	\end{align*}
	for the second equality we used that $\langle q, \Res(\lambda, S) e \rangle = 0$.
	
	We only need to show that the result of the preceding computation cannot be $1$,
	and to this end, it suffices to show that the modulus of 
	$\langle q, \Res(\lambda, S) q \rangle$ is strictly less than $1$.
	To see this, we once again use the block diagonal representation of $\Res(\lambda,S)$,
	together with the fact that $\Res(\lambda, (1-q_n)P)\one = \frac{1}{\lambda - 1 + q_n} \one$
	since $P \one = \one$.
	Thus, we obtain
	\begin{align*}
		\modulus{\langle q, \Res(\lambda, S) q \rangle}
		& = 
		\modulus{\sum_{n=1}^\infty \big\langle q_n \one \, , \, \Res(\lambda, (1-q_n)P) q_n \one \big\rangle }
		\\
		& =
		\modulus{\sum_{n=1}^\infty q_n^2 \frac{d}{\lambda - 1 + q_n}}
		\le \sum_{n=1}^\infty q_n d \, \frac{q_n}{\modulus{\lambda - 1 + q_n}}.
	\end{align*}
	The numbers $q_n d$ sum up to $1$, 
	and the numbers $\frac{q_n}{\modulus{\lambda - 1 + q_n}}$ are all strictly less than $1$
	according to Proposition~\ref{prop:circles}.
	Hence, $\modulus{\langle q, \Res(\lambda, S) q \rangle} < 1$,
	which shows that $\langle q, \Res(\lambda,S + e \otimes q) e \rangle \not= 1$.
	Therefore, $\lambda$ is in the resolvent set of $T$.
\end{proof}

\section{A $C_0$-semigroup version}
\label{section:semigroup-version}

In this section we present an analogous construction for the case of $C_0$-semigroups.
Throughout the section we freely make use of $C_0$-semigroup theory;
standard references for this topic include \cite{Pazy1983, EngelNagel2000}.
Perron--Frobenius type results for positive $C_0$-semigroups can, for instance,
be found in \cite[Chapters~B-III and~C-III]{Nagel1986} 
or in \cite[Chapters~12 and~14]{BatkaiKramarFijavzRhandi2017}.
As a motivation, we first observe the following property:

\begin{proposition}
	\label{prop:pnt-spec-semigroups}
	Let $(T(t))_{t \in [0,\infty)}$ be a bounded, positive and irreducible
	$C_0$-semi\-group on a Banach lattice $E$. 
	Then the set
	\begin{align*}
		\pntSpec(A) \cap i \bbR
	\end{align*}
	is either empty or an (additive) subgroup of $i\bbR$.
\end{proposition}

For the single operator case, we quoted an analogous result from
\cite[Theorem~V.5.2 and Lemma~V.4.8]{Schaefer1974} in the previous section.
For $C_0$-semigroups, we found this result 
only under slightly different assumptions in the literature, 
so we include a brief argument here:

\begin{proof}[Proof of Proposition~\ref{prop:pnt-spec-semigroups}]
	Assume that $\pntSpec(A) \cap i \bbR$ is non-empty 
	and let $i\beta$ be a point in this set;
	let $z \in E$ denote a corresponding eigenvector.
	Then, by the positivity of the semigroup, 
	we have $\modulus{z} = \modulus{T(t)z} \le T(t) \modulus{z}$
	for each time $t \ge 0$. 
	So $\modulus{z}$ is a non-zero \emph{super fixed vector} of the semigroup
	and hence, we can use the same argument as in \cite[Lemma~V.4.8]{Schaefer1974}
	to conclude that the dual semigroup has a non-zero positive fixed vector.
	
	Under this assumption, the subgroup property of $\pntSpec(A) \cap i \bbR$
	is proved in \cite[Theorem~C-III-3.8(a)]{Nagel1986}
	or in \cite[Proposition~14.15(c)]{BatkaiKramarFijavzRhandi2017}.
\end{proof}

Similarly as in the previous section, we now show 
that the assertion of Proposition~\ref{prop:pnt-spec-semigroups} is not true, 
in general, for the spectrum instead of the point spectrum.
In order to keep the technicalities as simple as possible,
we will not prove the result in the same generality as Theorem~\ref{thm:main-result};
instead, we will restrict ourselves to the construction of a single semigroup
with a certain spectral property:

\begin{theorem}
	\label{thm:semigroup-case}
	There exists an $L^1$-space $E$ over a $\sigma$-finite measure space
	and an irreducible and stochastic $C_0$-semigroup on $E$
	such that the spectrum $\spec(A)$ of its generator $A$ satisfies
	\begin{align*}
		\spec(A) \cap i \bbR = i (-\infty,-1] \; \cup \; \{0\} \; \cup \; i [1,\infty).
	\end{align*}
\end{theorem}

This theorem is -- in a negative sense -- relevant for the analysis of the long-term behaviour of
positive $C_0$-semigroups:

\begin{remark}
	If a $C_0$-semigroup is bounded and the spectrum of its generator $A$ intersects $i\bbR$
	in at most countably many points, this implies 
	-- under a certain ergodicity assumption --
	that the semigroup is asymptotically almost periodic;
	this is a version of the ABLV theorem, see for instance
	\cite[Theorem~5.5.5]{ArendtBattyHieberNeubrander2011}.
	So it is natural to ask for sufficient conditions of $\spec(A) \cap i\bbR$
	to be at most countable.
	
	If irreducibility of a positive $C_0$-semigroup implied that $\spec(A) \cap i\bbR$ is a group 
	then, for such a $C_0$-semigroup, it would suffice to know that at least one number in $i\bbR$
	is not a spectral value in order to conclude that $\spec(A) \cap i\bbR$ is at most countable
	(since all closed non-trivial subgroups of $i\bbR$ are countable).

	Theorem~\ref{thm:semigroup-case}, though, shows that such an argument cannot work,
	in general.
\end{remark}

Our construction of the semigroup generator $A$ in Theorem~\ref{thm:semigroup-case}
is quite similar to what we did in the previous section. 
Here are the details:

\begin{construction}
	\label{construction:semigroup}
	\emph{The space:}
	Endow the complex unit circle $\bbT$ with the Haar measure 
	that assigns the measure $2\pi$ to the whole space $\bbT$.
	We denote the norm on $L^1(\bbT)$ by $\norm{\argument}_1$, 
	and we set
	\begin{align*}
		E := \{f = (f_0, f_1, f_2, \dots): \; f_0 \in \bbC, \; f_1,f_2, \dots \in L^1(\bbT) \text{ and } \norm{f}_E < \infty\},
	\end{align*}
	where
	\begin{align*}
		\norm{f}_E := \modulus{f_0} + \sum_{n=1}^\infty \norm{f_n}_1.
	\end{align*}
	Clearly, $E$ is isometrically lattice isomorphic to the $L^1$-space
	over the $\sigma$-finite measure space
	$\{0\} \, \dot\cup \, \bbT \, \dot\cup \, \bbT \, \dot\cup \, \dots$.
	
	\emph{The generator:} 
	We specify our semigroup by defining its generator $A$. 
	Let $D: L^1(\bbT) \supseteq \dom{D} \to L^1(\bbT)$ denote 
	the generator of the shift semigroup on $L^1(\bbT)$;
	then the spectrum of $D$ equals $i\bbZ$. 
	
	Choose a sequence $(q_n)_{n \in \bbN} \subseteq (0,1]$ such that
	$\sum_{n=1}^\infty 2\pi q_n = 1$; 
	clearly, $q_n$ converges to $0$ as $n \to \infty$.
	In addition, let $(\omega_n)_{n \in \bbN} \subseteq [1,2] \cap \bbQ$
	be a sequence which contains each rational number in $[1,2]$ infinitely often.
	We first define a block diagonal operator $B: E \supseteq \dom{B} \to E$ as
	\begin{align*}
		B =
		\begin{pmatrix}
			-1 &                  &                  &        \\
			   & \omega_1 D - q_1 &                  &        \\
			   &                  & \omega_2 D - q_2 &        \\
			   &                  &                  & \ddots
		\end{pmatrix};
	\end{align*}
	its domain is given by
	\begin{align*}
		\dom{B} := 
		\big\{
			f = (f_0, f_1, f_2, \dots): \; 	& f_0 \in \bbC, \; f_1,f_2, \dots \in \dom{D} 
											\\
		                                	& \text{ and } 
		                                	  \sum_{n=1}^\infty \norm{D f_n}_1 < \infty
		\big\}.
	\end{align*}
	By standard perturbation theory, $B$ is the generator of a positive semigroup on $E$.
	Now, let $\one \in L^1(\bbT)$ denote the constant function with value $1$
	(it has norm $2\pi$)
	and consider the vectors
	\begin{align*}
		e = (1,0,0,\dots) 
		\qquad \text{and} \qquad 
		q = (0, q_1\one, q_2\one, \dots);
	\end{align*}
	we can interpret both of them as vectors in $E$ 
	and as vectors in the dual space $E'$.
	Similarly as in Construction~\ref{construction:single-operator}
	we define the operator $A: E \supseteq \dom{A} \to E$ by
	\begin{align*}
		A := B  +  e \otimes q  +  q \otimes e,
	\end{align*}
	with domain $\dom{A} := \dom{B}$.
\end{construction}

Clearly, $A$ generates a positive $C_0$-semigroup on $E$,
and a straightforward computation shows that the vector
\begin{align*}
	(1, \one, \one, \dots) \in E'
\end{align*}
is in the kernel of the dual operator $A'$
(one merely has to use that $\one \in \ker D'$).
Hence, the semigroup generated by $A$ is stochastic.

Moreover, one readily checks that the rank-$2$ operator 
$e \otimes q  +  q \otimes e$ is irreducible. 
Thus, the semigroup generated by $A$ is irreducible, too 
\cite[Proposition~C-III-3.3]{Nagel1986}.

So in order to obtain Theorem~\ref{thm:semigroup-case},
it only remains to compute the peripheral spectrum of $A$.
We start with the easier of both inclusions:

\begin{proposition}
	\label{prop:peripheral-spectrum-of-generator-easy}
	We have
	\begin{align*}
		\spec(A) \cap i \bbR \supseteq i (-\infty,-1] \cup \{0\} \cup i [1,\infty).
	\end{align*}
\end{proposition}
\begin{proof}
	As the semigroup generated by $A$ is stochastic, $0$ is a spectral value of $A$.
	Now, let $r$ be a rational number in $[1,\infty)$.
	It suffices to prove that $ir \in \spec(A)$
	(since the spectrum is closed and invariant under complex conjugation).
	Let $k \in \bbN$ denote the largest integer that is smaller than $r$.
	Then there exists a rational number $\omega \in [1,2)$ such that $\omega k = r$.
	As the number $ik$ is an eigenvalue value of the differential operator $D$,
	it follows that $i \omega k = ir$ is an eigenvalue of $\omega D$.
	
	Since the operator $\omega D$ occurs, up to the perturbations $-q_n$,
	infinitely many often in the definition of $B$,
	one can use the same argument as in the proof of
	Proposition~\ref{prop:approximate-eigenvalue} to show that $ir$ is an approximate eigenvalue of $A$.
\end{proof}

We still have to show the converse inclusion for the peripheral spectrum,
and to this end we proceed similarly as in the single operator case:
we start with the operator $B$, and then we use the Sherman--Morrison formula.

\begin{proposition}
	\label{prop:spectrum-of-B}
	Let $\beta \in (-1,1) \setminus \{0\}$.
	Then $i\beta$ is in the resolvent set of $B$.
\end{proposition}
\begin{proof}
	Since the spectrum of $D$ equals $i\bbZ$, 
	we can find a number $M > 0$ such that
	\begin{align*}
		\norm{\Res(\lambda,D)} \le M
	\end{align*}
	for all $\lambda$ in the rectangle $[0,1] + i\beta [\frac{1}{2},1]$.
	Thus, we obtain
	\begin{align*}
		\norm{\Res(i\beta, \omega_n D - q_n)}
		=
		\frac{1}{\omega_n}\norm{\Res\big(\frac{i\beta + q_n}{\omega_n}, D\big)} 
		\le 
		M
	\end{align*}
	for all $n \in \bbN$ 
	(as we assumed the numbers $\omega_n$ to be in $[1,2]$
	and the numbers $q_n$ to be in $(0, 1]$). 
	Hence, the resolvents $\Res(i\beta, \omega_n D - q_n)$ are uniformly bounded as $n$ varies.
	Due to the block diagonal form of $B$ this implies that $i\beta$ is in the resolvent set of $B$.
\end{proof}

Now we can finally prove that the spectrum of $A$ is the claimed set,
again by employing the Sherman--Morrison type result
from Proposition~\ref{prop:sherman-morrison-woodbury} twice:

\begin{proposition}
	\label{prop:resolvent-set-on-imaginary-axis}
	Let $\beta \in (-1,1) \setminus \{0\}$.
	Then $i\beta$ is in the resolvent set of $A$.
\end{proposition}
\begin{proof}
	We already know from Proposition~\ref{prop:spectrum-of-B}
	that $i\beta$ is in the resolvent set of $B$.
	Let us first note that it is also in the resolvent set of $B + e \otimes q$:
	indeed, we have $\langle e, \Res(i\beta,B) q \rangle = 0$,
	so Proposition~\ref{prop:sherman-morrison-woodbury} tells us that
	$i\beta$ is in the resolvent set of $B + e \otimes q$, 
	and that the resolvent of this operator at $i\beta$ is given by
	\begin{align*}
		\Res(i\beta, B + e \otimes q)
		=
		\Res(i\beta, B) + \Res(i\beta, B) (e \otimes q) \Res(i\beta, B).
	\end{align*}
	In order to prove that $i\beta$ is also in the resolvent set of
	$A = B + e \otimes q + q \otimes e$, we have to show,
	according to Proposition~\ref{prop:sherman-morrison-woodbury},
	that $\langle q, \Res(i\beta, B + e \otimes q) e \rangle \not= 1$.
	So let us compute
	\begin{align*}
		\langle q, \Res(i\beta, B + e \otimes q) e \rangle
		& =
		\langle q, \Res(i\beta, B) e \rangle
		+ 
		\langle q, \Res(i\beta, B) (e \otimes q) \Res(i\beta, B) e \rangle
		\\ 
		& =
		\langle e, \Res(i\beta, B) e \rangle  \cdot  \langle q, \Res(i\beta, B) q \rangle
		= 
		\frac{1}{i\beta + 1} \cdot  \langle q, \Res(i\beta, B) q \rangle;
	\end{align*}
	for the second equality we used that $\langle q, \Res(i\beta, B) e \rangle = 0$.
	In order to further simplify the expression, 
	we note that $D\one = 0$, 
	so $\Res(i\beta, \omega_n D - q_n) \one = \frac{1}{i\beta + q_n}\one$ for each $n \in \bbN$.
	Therefore,
	\begin{align*}
		\langle q, \Res(i\beta, B) q \rangle
		= 
		\sum_{n=1}^\infty 2\pi q_n \frac{q_n}{i\beta + q_n}.
	\end{align*}
	The numbers $2\pi q_n$ sum up to $1$ and 
	the modulus of $\frac{q_n}{i\beta + q_n}$ is strictly less than $1$
	for each $n$. 
	Hence, $\modulus{\langle q, \Res(i\beta, B) q \rangle} < 1$.
	
	The modulus of $\frac{1}{i\beta + 1}$ is also strictly less than $1$, so
	\begin{align*}
		\modulus{\langle q, \Res(i\beta, B + e \otimes q) e \rangle}
		=
		\modulus{\frac{1}{i\beta + 1}} \cdot \modulus{\langle q, \Res(i\beta, B) q \rangle}
		<
		1. 
	\end{align*}
	In particular, the number $\langle q, \Res(i\beta, B + e \otimes q) e \rangle$ is not $1$,
	so $i\beta$ is indeed in the resolvent set of $A$.
\end{proof}

\subsection*{Acknowledgements} 

It is my pleasure to thank Wolfgang Arendt 
for a very interesting discussion
that motivated the results presented in this note,
and to thank Ulrich Groh 
for pointing out Schaefer's result in \cite[Theorem~7]{Schaefer1968} to me.

%
%
% ...

\bibliographystyle{plain}
\bibliography{literature}

\end{document}